\documentclass{article}
\usepackage{amsmath,amssymb,amsthm}
\usepackage{bbm}            
\usepackage{bm}             

\def\D#1{D_{\!#1}}
\def\d{\partial}
\def\Z{\mathbbm Z}
\def\R{\mathbbm R}
\def\C{\mathbbm C}
\def\rmd{\,\mathrm{d}}  
\def\0{{\bm 0}}
\def\a{{\bm a}}
\def\e{{\bm e}}
\def\k{{\bm k}}
\def\s{{\bm s}}
\def\x{{\bm x}}
\def\y{{\bm y}}
\renewcommand{\leq}{\leqslant}
\renewcommand{\geq}{\geqslant}

\newtheorem{theorem}{Theorem}

\newtheorem{corollary}[theorem]{Corollary}

\theoremstyle{definition}
\newtheorem{example}{Example}

\newcommand\xqed[1]{%
  \leavevmode\unskip\penalty9999 \hbox{}\nobreak\hfill
  \quad\hbox{#1}}
\newcommand\exampleEnd{\xqed{$\qed$}}

\begin{document}

\title{\vskip -10mm Lattice Green's Functions of the Higher-Dimensional Face-Centered Cubic Lattices}

\author{Christoph~Koutschan\\        
  $\!$Johann Radon Institute for Computational and Applied Mathematics$\!$\\
  Austrian Academy of Sciences\\
  Altenberger Stra\ss e 69, A-4040 Linz, Austria}

\maketitle

\begin{abstract}
  We study the face-centered cubic lattice (fcc) in up to six dimensions.  In
  particular, we are concerned with lattice Green's functions (LGF) and return
  probabilities. Computer algebra techniques, such as the method of creative
  telescoping, are used for deriving an ODE for a given LGF. For the four- and
  five-dimensional fcc lattices, we give rigorous proofs of the ODEs that were
  conjectured by Guttmann and Broadhurst. Additionally, we find the ODE of the
  LGF of the six-dimensional fcc lattice, a result that was not believed to be
  achievable with current computer hardware.
  \footnote{This article appeared in J. Phys. A: Math. Theor. \textbf{46} (2013) 125005.\\
    DOI: 10.1088/1751-8113/46/12/125005,
    http:/$\!$/iopscience.iop.org/1751-8121/46/12/125005}
\end{abstract}


\vspace{0.5pc}

\section{Introduction}\label{sec.intro}

Random walks on lattices, such as the face-centered cubic lattice, are an
important concept in various applications in physics, chemistry, ecology,
economics, and computer science, when lattice vibration problems (phonons),
diffusion models, luminescence, Markov processes and other random processes
are studied. 
A fundamental object to investigate is the probability generating function of
a lattice, called the lattice Green's function (LGF).  For example, the return
probability of a lattice can be expressed in terms of the LGF. The LGFs of
three-dimensional lattices have been computed and analyzed
in~\cite{Joyce98,Guttmann09}, for higher-dimensional lattices
see~\cite{Broadhurst09,Guttmann10}. We present a completely different approach
to LGFs that is based on computer algebra techniques, with which we are not
only able to confirm independently the previously known results, but also go
beyond. We believe that this methodology can be applied successfully to many
other, yet unsolved problems of similar flavor, and therefore should be
popularized in the community.

The paper is organized as follows. In Section~\ref{sec.intro} we explain the
general setting in which we work and introduce basic notions, such as Bravais
lattice (Section~\ref{sec.bravais}), random walk (Section~\ref{sec.random}),
and lattice Green's function (Section~\ref{sec.lgf}); in
Section~\ref{sec.diffeq} an approach to LGFs via differential equations is
motivated. We derive the integral representation of the LGF and show how a
differential equation is connected to it. Section~\ref{sec.em} is dedicated to
different methods how to compute the ODE of a LGF in a nonrigorous way:
Section~\ref{sec.intrep} reviews the method for computing Taylor coefficients
used in~\cite{Broadhurst09,Guttmann10} and Sections~\ref{sec.count}
and~\ref{sec.multi} count random walks in order to obtain sufficient data to
construct the ODE. The main contribution of our work is contained in
Section~\ref{sec.ca}.  The method of creative telescoping is described in
Section~\ref{sec.ct}; it enables us to compute the desired ODEs in a
mathematical rigorous way, including correctness certificates. Applying this
method to the LGF of the fcc lattice confirms the results
of~\cite{Broadhurst09,Guttmann10} in dimensions four and five, and yields an
ODE for the LGF of the six-dimensional fcc lattice that was not known
previously (see Section~\ref{sec.results}).

\subsection{Bravais Lattices}\label{sec.bravais}

We consider lattices in $\R^d$ that are given as infinite
sets of points
\[
  \bigg\{\sum_{i=1}^d n_i\a_i: n_1,\dots,n_d\in\Z\bigg\} \subseteq\R^d
\]
for some linearly independent vectors $\a_1,\dots,\a_d\in\R^d$
(throughout this paper, vectors are denoted by bold letters).  In
three dimensions such lattices are called \emph{Bravais lattices}.
The simplest instance of such a lattice is obtained by choosing
$\a_i=\e_i$, the $i$-th unit vector; the result is the integer
lattice~$\Z^d$ which is also called the \emph{square lattice} (for
$d=2$), or the \emph{cubic lattice} (for $d=3$), or the
\emph{hypercubic lattice} (for $d\geq 4$).

The \emph{face-centered (hyper-) cubic lattice} (fcc lattice) is
obtained from the (hyper-) cubic lattice by adding the center point of
each (two-dimensional) face to the set of lattice points. In two
dimensions this operation is trivial: the faces of the square lattice
$\Z^2$ are all unit squares with corners
$(m,n),(m+1,n),(m+1,n+1),(m,n+1)$ for integers $m,n\in\Z$. Their
center points are $\Z^2+(\frac12,\frac12)$ which together with $\Z^2$
again yields a square lattice, more precisely a copy of $\Z^2$ which
is rotated by 45 degrees and shrunk by a factor of $\sqrt{2}$. The
situation becomes more interesting in higher dimensions. For example,
in three dimensions there are 6 faces of the unit cube, and their
center points together with all integral translates have to be
included.  It is not difficult to see that the three-dimensional fcc
lattice consists of four copies of $\Z^3$, namely
\[
  \textstyle\Z^3
    \cup\left(\Z^3+\left(\frac12,\frac12,0\right)\right)
    \cup\left(\Z^3+\left(\frac12,0,\frac12\right)\right)
    \cup\left(\Z^3+\left(0,\frac12,\frac12\right)\right).
\]
Similarly the fcc lattice in four dimensions consists of 7 copies of
$\Z^4$, and in general the $d$-dimensional fcc lattice is composed of
$1+\binom{d}{2}$ translated copies of $\Z^d$.  

The study of Bravais lattices was inspired by crystallography in as much as
the atomic structure of crystals forms such regular lattices. While
the cubic lattice is quite rarely found in nature (e.g., in polonium)
due to its small \emph{atomic packing factor} (the proportion of space
that is filled when a sphere of maximal radius is put on each lattice
point, in a way that these spheres do not overlap), the fcc lattice is
more often encountered, for example, in aluminium, copper, silver, and
gold.  The atomic packing factor of the fcc lattice is
$\sqrt{2}\pi/6$, the highest possible value as was shown by 
Hales in his famous proof of the Kepler conjecture~\cite{Hales05}.

\subsection{Random Walks}\label{sec.random}
For the sake of simplicity, the fcc lattice as introduced in the previous
section, is stretched by a factor of~$2$ in all coordinate directions so that
all lattice points have integral coordinates. This convention is kept
throughout the paper as it does not change the relevant quantities that we are
interested in (e.g., the return probability, see below).

The aim of this paper is to study random walks on the fcc lattice in
several dimensions. We consider walks that allow only steps to the
nearest neighbors of a point (with respect to the Euclidean metric).
Furthermore it is assumed that all steps are taken with the same
probability. For example, consider a point $(k,m,n)$ in the
three-dimensional cubic lattice $(2\Z)^3$. It is the common corner
point of 8 cubes.  The nearest neighbors in the 3D fcc lattice are
then the center points of some of those faces which have $(k,m,n)$ as a
corner point. Note that they all have distance $\sqrt{2}$ whereas the
other corner and face-center points are farther away (their distance
is $\geq2$) and hence not reachable in a single step. Thus the number
of possible steps is $8\cdot 3/2=12$ (number of adjacent cubes times
the number of adjacent faces per cube, divided by two since each face
belongs to two cubes).  The same situation is encountered at the
center point of some face and hence every point in the 3D fcc lattice
has exactly 12 nearest neighbors; this number is called the
\emph{coordination number} of the lattice.

The above considerations can be generalized to arbitrary dimensions in
a straight-forward manner; one finds that the set of permitted steps
in the $d$-dimensional fcc lattice is given by
\begin{equation}\label{eq.steps}
  \textstyle\left\{(s_1,\dots,s_d)\in\{0,-1,1\}^d:|s_1|+\dots+|s_d|=2\right\}
\end{equation}
and thus its coordination number is $4\binom{d}{2}$.

\subsection{Lattice Green's Function}\label{sec.lgf}
Let $p_n(\x)$ denote the probability that a random walk which started
at the origin~$\0$ ends at point~$\x$ after $n$~steps.  Note that in
our setting of unrestricted walks, $c^np_n(\x)$ is an integer and
gives the total number of walks that end at location~$\x$ after
$n$~steps, where $c$ is the coordination number of the lattice.

In order to achieve information about random walks on the fcc lattice,
the following multivariate generating function is introduced:
\begin{equation}\label{eq.gfx}
  P(\x;z) = \sum_{n=0}^\infty p_n(\x)z^n.
\end{equation}
This function is called the \emph{lattice Green's function} (LGF).  
By defining the \emph{structure function}
$\lambda(\k)=\lambda(k_1,\dots,k_d)$ of a lattice to be the discrete
Fourier transform
\[
  \lambda(\k) = \sum_{\x\in\R^d} p_1(\x)e^{i\x\cdot\k}
\]
of the single-step probability function~$p_1(\x)$, the generating 
function~\eqref{eq.gfx} can be expressed as the $d$-dimensional integral
\[
  P(\x;z) = \frac{1}{\pi^d}\int_0^\pi\dots\int_0^\pi 
    \frac{e^{i\x\cdot\k}}{1-z\lambda(\k)}\rmd k_1\dots\rmd k_d.
\]
We shall be interested in walks which return to the origin and which
we therefore call \emph{excursions}. The LGF for excursions is given
by
\begin{equation}\label{eq.lgf}
  P(\0;z) = \sum_{n=0}^\infty p_n(\0)z^n = 
            \frac{1}{\pi^d}\int_0^\pi\dots\int_0^\pi
              \frac{\rmd k_1\dots\rmd k_d}{1-z\lambda(\k)}.
\end{equation}
In the following, we will only refer to this special instance when talking
about LGFs.  This function allows one to calculate the \emph{return
  probability}~$R$, sometimes also referred to as the \emph{P\'{o}lya number},
of the lattice.  It signifies the probability that a random walk that started
at the origin will eventually return to the origin. It can be computed via the
formula
\begin{equation}
  R = 1-\frac{1}{P(\0;1)} = 1-\frac{1}{\sum_{n=0}^{\infty} p_n(\0)}.
\end{equation}
\begin{example}\label{ex.sq}
Consider the square lattice~$\Z^2$ which admits the steps $(-1,0)$,
$(1,0)$, $(0,-1)$, and $(0,1)$. Its structure function is
\[
  \textstyle\lambda(k_1,k_2) = \frac14\left(e^{-ik_1}+e^{ik_1}+e^{-ik_2}+e^{ik_2}\right)
                   = \frac12\left(\cos k_1+\cos k_2\right).
\]
and therefore its LGF is (see, e.g., \cite{Guttmann10})
\[
  P(0,0;z) = \frac{1}{\pi^2} \int_0^\pi\int_0^\pi\frac{\rmd k_1\rmd k_2}
               {1-\frac{z}{2}\left(\cos k_1+\cos k_2\right)}
           = \frac{2}{\pi}\mathbf{K}(z^2)
\]
where $\mathbf{K}(z)$ is the complete elliptic integral of the first
kind. The fact that the above integral diverges for $z=1$ immediately
implies that the return probability $R=1$; in other words that every
random walk will eventually return to the origin, a result that was
already proven in 1921 by P\'{o}lya~\cite{Polya21}.  
\exampleEnd
\end{example}

\begin{example}\label{ex.fcc2}
We have already remarked that the two-dimensional fcc lattice is
nothing else but a rotated and stretched version of the square
lattice.  Nevertheless let's have a look at the LGF when the step set
$\{(-1,-1),(-1,1),(1,-1),(1,1)\}$ is taken. Its structure function is
\begin{eqnarray*}
  \lambda(k_1,k_2) 
    & = & \textstyle\frac14\big(e^{-i(k_1+k_2)}+e^{-i(k_1-k_2)}+e^{i(k_1-k_2)}+e^{i(k_1+k_2)}\big)\\
    & = & \textstyle\frac12\big(\cos(k_1+k_2)+\cos(k_1-k_2)\big) = \cos k_1\cos k_2,
\end{eqnarray*}
using the well-known angle-sum identity 
$\cos(x \pm y) = \cos x\cos y \mp \sin x\sin y$.
Although the structure function differs from that in
Example~\ref{ex.sq} the LGF is the same:
\begin{equation}\label{eq.lgf2d}
  P(0,0,z) = \frac{1}{\pi^2} \int_0^\pi\int_0^\pi\frac{\rmd k_1\rmd k_2}{1-z\cos k_1\cos k_2}
           = \frac{2}{\pi}\mathbf{K}(z^2)
\end{equation}
as shown in Equation (6) of~\cite{Guttmann10}.  Note also that the
different distances between nearest-neighboring lattice points---$1$
in Example~\ref{ex.sq} and $\sqrt{2}$ in Example~\ref{ex.fcc2}---carry
no weight since only excursions (and not walks with arbitrary end
points) are investigated.
\exampleEnd
\end{example}

It is now an easy exercise to compute the structure function~$\lambda(\k)$
for the $d$-dimensional fcc lattice:
\[
  \lambda(\k) = \binom{d}{2}^{-1}\sum_{m=1}^d\sum_{n=m+1}^d \cos k_m\cos k_n.
\]
The LGF is then given as the $d$-fold integral~\eqref{eq.lgf}
and the return probability can be computed by integrating over
$1/(1-\lambda(\k))$

\subsection{The Differential Equation Detour}\label{sec.diffeq}
The return probability in the fcc lattice in three dimensions was
first computed by Watson~\cite{Watson39} as one of the three
integrals which were later named after him, and which give the return
probabilities in different three-dimensional lattices. These
probabilities can be expressed in terms of algebraic numbers, $\pi$,
and values of the Gamma function at rational arguments. For example, the
probability of returning to the origin in the 3D fcc lattice is given
by
\[
  1 - \frac{16\sqrt[3]{4}\pi^4}{9\big(\Gamma(\frac13)\big)^6}.
\]
For the three-dimensional fcc lattice, Joyce expressed the lattice Green's
function also in terms of complete elliptic integrals~\cite{Joyce98}, but the
expression is fairly complicated and for the higher-dimensionsal fcc lattices
no such evaluation is known at all.  Similarly we don't know of any
closed-form representation of the return probabilities in higher dimensions.

Instead we will derive differential equations for the corresponding
LGFs. Although less explicit than the previously mentioned closed-form
results, such an implicit representation of the LGF provides considerable 
insight. It allows one to compute the number of excursions efficiently
for any fixed number of steps, as well as the return
probability with very high precision (see Section~\ref{sec.ca}). But
also the differential equations themselves reveal very interesting
properties that are worth investigation.

To motivate our approach and to illuminate the origin of
Equation~\eqref{eq.lgf}, consider an arbitrary lattice in~$\Z^d$ with
some finite set~$S\subset\Z^d$ of permitted steps. Then clearly the
probability function~$p_n(\x)$ satisfies the constant-coefficient
recurrence
\begin{equation}\label{eq.trivrec}
  p_{n+1}(\x) = \frac{1}{|S|}\sum_{\s\in S}p_n(\x-\s).
\end{equation}
Let~$F(\y;z)$ denote the multivariate generating function 
\[
  F(\y;z)=\sum_{n=0}^\infty\sum_{\x\in\Z^d}p_n(\x)\y^{\x}z^n.
\]
Multiplying both sides of~\eqref{eq.trivrec} by~$\y^{\x}z^n$
and summing with respect to $n$ and $\x$ gives
\begin{eqnarray*}
\sum_{n=0}^{\infty}\sum_{\x\in\Z^d}p_{n+1}(\x)\y^{\x}z^n & = &
  \frac{1}{|S|}\sum_{n=0}^{\infty}\sum_{\x\in\Z^d}\sum_{\s\in S}p_n(\x-\s)\y^{\x}z^n\\
\frac{1}{z}\sum_{n=1}^{\infty}\sum_{\x\in\Z^d}p_n(\x)\y^{\x}z^n & = &
  \frac{1}{|S|}\sum_{\s\in S}\sum_{n=0}^{\infty}\sum_{\x\in\Z^d}p_n(\x)\y^{\x+\s}z^n\\
\frac{1}{z}\left(F(\y;z)-1\right) & = &
  \frac{1}{|S|}\sum_{\s\in S}\y^{\s}F(\y;z)
\end{eqnarray*}
Thus we obtain
\[
  F(\y;z) = \frac{1}{1-\frac{z}{|S|}\sum_{\s\in S}\y^{\s}}
\]
and the LGF $P(\0;z)$ is nothing
else but the constant term $\langle\y^{\0}\rangle F(\y;z)$. A differential
equation for this expression can be derived from an operator of the form
\begin{equation}\label{eq.ctmult}
  A(z,\D{z}) + \D{y_1}B_1 + \dots + \D{y_d}B_d
\end{equation}
that annihilates the expression $F(\y;z)/(y_1\cdots y_d)$.
Here the symbol~$\D{x}$ denotes the partial derivative w.r.t.~$x$
and the $B_j$ are differential operators that may involve 
$y_1,\dots,y_d,z$ as well as $\D{y_1},\dots,\D{y_d},\D{z}$.
The fact that~$A$ may only depend on $z$ and $\D{z}$ is crucial
and therefore explicitly indicated. In Section~\ref{sec.ca}
we will discuss how to find such an operator. From
\[
  \big\langle y_1^{-1}\cdots y_d^{-1}\big\rangle A(z,\D{z})\frac{F(\y,z)}{y_1\cdots y_d}\> +\>
  \sum_{j=1}^d \big\langle y_1^{-1}\cdots y_d^{-1}\big\rangle \D{y_j}B_j\frac{F(\y;z)}{y_1\cdots y_d} = 0
\]
and the fact that the coefficient of $y^{-1}$ in an expression of
the form $\D{y}\sum_{n=-\infty}^\infty a_ny^n$ is always zero, it follows
that $A\big(\langle\y^{\0}\rangle F(\y;z)\big)=A\big(P(\0;z)\big)=0$.
Also in Section~\ref{sec.ca} we will demonstrate how the
operator~\eqref{eq.ctmult} is used to derive a differential equation
for the $d$-fold integral
\[
  \int\cdots\int \frac{d\y}{(y_1\cdots y_d)(1-\frac{z}{|S|}\sum_{\s\in S}\y^{\s})} = 
  \int\cdots\int \frac{d\k}{1-z\lambda(\k)}.
\]

\section{An Experimental Mathematics Approach}\label{sec.em}
This section presents some results that were obtained in a
non-rigorous way using the method of guessing~\cite{Kauers09}. 
That is, for finding a linear differential equation
\[
  c_m(x)f^{(m)}(x)+\dots+c_1(x)f'(x)+c_0(x)f(x)=0
\]
for a certain function~$f(x)$, one computes the first terms of the Taylor
expansion of~$f(x)$ and then makes an ansatz with undetermined polynomial
coefficients~$c_j(x)$. If the resulting linear system is overdetermined (i.e.,
if sufficiently many Taylor coefficients were used) but still admits a
nontrivial solution, then the detected ODE is very likely to be correct.
Another strategy to gain confidence in the result, is to test it with further
Taylor coefficients, that were not used in the computation. However, this
method can never produce a rigorous proof of the result and there always
remains a (very small) probability that the guess is wrong. For this reason,
any result obtained in this fashion (e.g., the ODEs presented
in~\cite{Guttmann09,Broadhurst09}) is termed a \emph{conjecture}.

\subsection{Starting from the Integral Representation}\label{sec.intrep}
In this section we briefly recapitulate some previous work done by Broadhurst
and Guttmann, who used the integral representation~\eqref{eq.lgf} of the LGF
as a starting point.

Guttmann computed a differential equation for the LGF of the four-dimensional
fcc lattice~\cite{Guttmann09} (see also Theorem~\ref{thm.fcc4}): for this
purpose, the four-fold integral was rewritten as a double integral whose
integrand was expanded as a power series.  Term-by-term integration yielded a
truncated Taylor expansion of the LGF which allowed him to apply the method of
guessing.

Recently, Broadhurst had obtained an ODE for the LGF of the five-dimensional
fcc lattice~\cite{Broadhurst09} (see also Theorem~\ref{thm.fcc5}), a result
that required several days of PARI calculations. Broadhurst's strategy
consisted in expanding the integrand in~\eqref{eq.lgf} as a geometric series
$\sum_{n=0}^\infty \lambda(\k)^nz^n$ and in expanding $\lambda(\k)^n$ using
the multinomial theorem (which gives a $(m-1)$-fold sum if $m$ is the number
of summands in $\lambda(\k)$). The inner terms can now be integrated using
Wallis' formula
\[
  \int_0^\pi \cos(x)^{2n}\rmd x = \frac{\pi}{4^n}\binom{2n}{n}.
\]
The structure function of the 5D fcc lattice consists of
$10$~summands.  Thus the computation of the $n$-th Taylor coefficient
of~$P(\0;z)$ requires the evaluation of a $9$-fold sum, or in other
words, has complexity $O(n^9)$.

\subsection{Counting the walks}\label{sec.count}
A different way to crank out as many Taylor coefficients of the LGF as
necessary is to explicitly count all possible excursions with a certain number
of steps. Let $a_n(\x)$ be the number of walks from~$\0$ to~$\x$ with
$n$~steps and let $c$ denote the coordination number of the lattice, then the
lattice Green's function $P(\0;z)=\sum_{n=0}^\infty a_n(\0)(z/c)^n$, as we
have already remarked earlier. The values of the $(d+1)$-dimensional sequence
($d$~again denotes the dimensionality of the lattice) can be computed with the
recurrence~\eqref{eq.trivrec}. To~obtain the first $n$ Taylor coefficients
hence requires one to fill the $(d+1)$-dimensional array
$\big(a_m(\x)\big)_{0\leq m,x_1,\dots,x_d<n}$ with values (by symmetry it
suffices to consider the first octant only, which again by symmetry can be
restricted to the wedge $x_1\geq x_2\geq\dots\geq x_d$). Still, this has
complexity~$O(n^{d+1})$.  Further optimizations consist in cutting off the
regions where the sequence can be predicted to be zero (e.g. when $x_j>n$ for
some~$j$), and to truncate the $x_j$-coordinates at $n/2$ (since we are
interested in excursions, points too far from the origin~$\0$ are not
relevant). Although the complexity is better than before, computing the full
array of values can be quite an effort. Again, in the example of the 5D fcc
lattice, about $115$ Taylor coefficients are necessary to recover the
recurrence for $a_n(\0)$ (which then gives rise to the differential equation
of~$P(\0;z)$), and hence the full array contains about $115\cdot
58^5/5!\approx 6.3\cdot 10^8$ values! Fortunately we can do better.

\subsection{Multi-Step guessing}\label{sec.multi}
How can the recurrence for $a_n(\0)$ be computed without calculating
all the values of the multivariate sequence~$a_n(\x)$ in the
box~$[0,n]^{d+1}$ (or some slightly optimized version of it)?  In the
previous section, we first computed lots of data, then threw away most
of it, and did a single guessing step.  But the guessing can be done
in several steps which we call \emph{multi-step guessing}.  The method
is illustrated on the 5D fcc example. As before, we start with the
recurrence~\eqref{eq.trivrec} to crank out a moderate number of values
for the six-dimensional sequence~$a_n(x_1,\dots,x_5)$, namely in the
box~$[0,15]^6$, which takes about 30 seconds only. From this array, we
pick the values of~$a_n(x_1,x_2,x_3,0,0)$ which constitute a
four-dimensional sequence that we denote with~$b_n(x_1,x_2,x_3)$. The
data is now used to guess recurrences for this new function~$b$. One
of these recurrences is
\[
  \begin{array}{l}
  (n+1)b_n(x_1,x_2+3,x_3+1) - (n+1)b_n(x_1,x_2+1,x_3+3) +{}\\
  (n+1)b_n(x_1+1,x_2,x_3+3) - (n+1)b_n(x_1+1,x_2+3,x_3) +{}\\ 
  (n+1)b_n(x_1+1,x_2+3,x_3+4) - (n+1)b_n(x_1+1,x_2+4,x_3+3) -{}\\
  (n+1)b_n(x_1+3,x_2,x_3+1) + (n+1)b_n(x_1+3,x_2+1,x_3) -{}\\
  (n+1)b_n(x_1+3,x_2+1,x_3+4) + (n+1)b_n(x_1+3,x_2+4,x_3+1) +{}\\
  (n+1)b_n(x_1+4,x_2+1,x_3+3) - (n+1)b_n(x_1+4,x_2+3,x_3+1) +{}\\
  (x_2+2)b_{n+1}(x_1+1,x_2+2,x_3+3) - (x_3+2)b_{n+1}(x_1+1,x_2+3,x_3+2) -{}\\
  (x_1+2)b_{n+1}(x_1+2,x_2+1,x_3+3) + (x_1+2)b_{n+1}(x_1+2,x_2+3,x_3+1) +{}\\
  (x_3+2)b_{n+1}(x_1+3,x_2+1,x_3+2) - (x_2+2)b_{n+1}(x_1+3,x_2+2,x_3+1) = 0
  \end{array}
\]
which has the disadvantage that it does not allow us to compute the
values $b_n(0,0,0)$ since the leading coefficient vanishes;
unfortunately this phenomenon occurs frequently in this context. An
additional recurrence that does not suffer from this handicap is much
larger and therefore not reproduced here. Anyway, guessing these
recurrences can be done in less than a minute.

Now these recurrences can be used to compute more values for the
sequence $b_n(x_1,x_2,x_3)$ (in 30 seconds one can now go up to
$n=30$) which in turn are used to~guess recurrences
for~$b_n(x_1,x_2,0)$.  These latter recurrences allow to compute
$a_n(\0)=b_n(0,0,0)$ for $0\leq n\leq 115$ in about $2.5$
minutes. Voil\`{a}! The whole computation takes less than 5 minutes on
a modest laptop. Of course it is a matter of trial and error to
determine how many coordinates are set to~$0$ in each step (in the
above example, we did~$2$ in the first step, $1$ in the second, and
again~$2$ in the third step).

\section{A Computer Algebra Approach}\label{sec.ca}
Again, we want to emphasize that the results presented in the previous
section are certainly nice, but lack mathematical rigor. To
achieve ultimate confidence in their correctness we have to apply a
different method. One possible such method is called \emph{creative
  telescoping}, a short introduction of which is given in the
following section.  After that we are able to state our results in
the form of theorems.

\subsection{Creative Telescoping}\label{sec.ct}
This method has been popularized by Zeilberger in his seminal
paper~\cite{Zeilberger90}.  Since then it has been applied to
innumerable identities involving hypergeometric summations, multisums,
integrals of special functions, and various other kinds of problems.
The basic idea is very simple and we illustrate it on the example
of a definite integral $F(z)=\int_a^b f(x,z)\rmd x$. The main step in the
algorithm consists in finding a partial differential equation for
$f(x,z)$ that can be written in the form
\begin{equation}\label{eq.ct}
  \big(A(z,\D{z}) + \D{x}B(x,z,\D{x},\D{z})\big)(f(x,z)) = 0
\end{equation}
where the \emph{telescoper} $A\in\C(z)\langle\D{z}\rangle$ and the \emph{delta
  part} $B\in\C(x,z)\langle\D{x},\D{z}\rangle$ are differential operators,
with the previously introduced notation of $\D{x}$ being the partial
derivative w.r.t.~$x$. By $\C(z)\langle\D{z}\rangle$ we denote the
non-commutative Ore algebra that can be viewed as a polynomial ring in the
``variable''~$\D{z}$ with rational function coefficients in~$\C(z)$.  The
result of the algorithm is a (possibly) inhomogeneous linear ODE for the
integral~$F(z)$ that is obtained by integrating Equation~\eqref{eq.ct}:
\[
  A\big(F(z)\big) + \Big[B\big(f(x,z)\big)\Big]_{x=a}^{x=b}=0.
\]
In applications one frequently encounters the situation that the second part
vanishes, yielding a homogeneous ODE. This is because many integrals that
occur in practice, have \emph{natural boundaries}. With our study of LGFs, we
are in a similar situation: in Section~\ref{sec.diffeq} it was shown that the
telescoper of~\eqref{eq.ctmult} automatically annihilates the LGF.
Nevertheless, in some of the present cases we did do the additional (but
superfluous) check that the differential equation is indeed homogeneous by
plugging in the boundaries of the integral, and got confirmation.

If this method is applied to a one-dimensional integral with 
hyperexponential integrand (i.e., its logarithmic derivative
is a rational function), then it is called the Almkvist-Zeilberger
algorithm. Its summation counterpart is the celebrated Zeilberger
algorithm for hypergeometric summation. For our purposes we have
to generalize the input class for the integrand to the so-called
$\d$-finite holonomic functions: a function~$f(x_1,\dots,x_d)$ 
is called $\d$-finite if for each~$x_i$ there exists a linear ODE 
for~$f$ with respect to~$x_i$. If in addition $f$ is holonomic
(the definition of this notion is somewhat technical and is omitted
here) then the existence of creative telescoping operators like
\eqref{eq.ctmult} or \eqref{eq.ct} is guaranteed.

The first algorithm to compute~\eqref{eq.ct} for general $\d$-finite functions
(our examples fall into this class, too) was proposed in~\cite{Chyzak00}. It
can deal with single integrations only and thus has to be applied iteratively
for multiple integrals. Its main drawback is its complexity that makes it
impossible to apply it to the problems discussed in this
paper. In~\cite{Koutschan10c} we have developed a different approach to
compute~\eqref{eq.ct} which is much better suited for large examples involving
multidimensional integrals. In addition, it can directly deal with multiple
integrals by computing operators of the form~\eqref{eq.ctmult}, but in the
present context it turned out to be more efficient to do the integrations step
by step. Both algorithms are implemented in our Mathematica package
\texttt{HolonomicFunctions}~\cite{Koutschan10b}, and a detailed introduction
into the topic is given in~\cite{Koutschan09}. The following example
demonstrates how this method is applied to the previously discussed
two-dimensional lattice.
\begin{example}
Looking at the integrand of Equation~\eqref{eq.lgf2d} one realizes
that it is not $\d$-finite since no linear ODE with respect to~$k_1$
can be found (and analogously for~$k_2$). But by means of the simple
substitutions $\cos k_1\to x_1$ and $\cos k_2\to x_2$
we can overcome this trouble: the integral now reads
\begin{equation}\label{eq.int2d}
  P(z) = \frac{1}{\pi^2} \int_{-1}^1\int_{-1}^1\frac{\rmd x_1\rmd x_2}{(1-z x_1 x_2)\sqrt{1-x_1^2}\sqrt{1-x_2^2}}.
\end{equation}
Let $f(x_1,x_2,z)$ denote the above integrand; it is easily verified
that it is a $\d$-finite function. The three ODEs w.r.t. $x_1$, $x_2$, 
and $z$ are given by the operators
\begin{eqnarray*}
  G_1 & = & (x_1x_2z-1)D_{\!z}+x_1x_2,\\
  G_2 & = & (x_2^2-1)(x_1x_2z-1)D_{\!x_2}+(2x_1x_2^2z-x_1z-x_2),\\
  G_3 & = & (x_1^2-1)(x_1x_2z-1)D_{\!x_1}+(2x_1^2x_2z-x_1-x_2z),
\end{eqnarray*}
so that $G_i\big(f(x_1,x_2,z)\big)=0$ for $i=1,2,3$. In this example,
it is an easy exercise to check that the creative telescoping
operator
\begin{equation}\label{eq.ct2d}
  z(z^2-1)D_{\!z}^2+(3z^2-1)D_{\!z}+z\>+\>\D{x_1}\frac{x_2-x_1^2x_2}{x_1x_2z-1}\>+\>\D{x_2}\frac{x_2z-x_2^3z}{x_1x_2z-1}
\end{equation}
annihilates the integrand~$f$. Indeed, it can be written as a
linear combination 
\[
  \left(\frac{z(z^2-1)}{x_1x_2z-1}D_{\!z}+\frac{x_1x_2z(z^2+1)-3z^2+1}{(x_1x_2z-1)^2}\right)G_1
  -\frac{x_2}{(x_1x_2z-1)^2}(zG_2+G_3)
\]
of the previously computed operators. It follows that
the double integral~\eqref{eq.int2d} satisfies the ODE
\[
  z(z^2-1)P''(z)+(3z^2-1)P'(z)+zP(z) = 0
\]
whose solution is the elliptic integral~$\mathbf{K}(z^2)$.

Alternatively, the two integrations can be performed in two steps (the
strategy that will be applied to the higher-dimensional fcc lattices).
In the first step (integration w.r.t.~$x_1$) the following two
creative telescoping operators are found:
\begin{eqnarray*}
  &&(x_2^2z^2-1)D_{\!z}+x_2^2z\>+\>\D{x_1}(x_1^2-1)x_2\\
  &&(x_2^2-1)(x_2^2z^2-1)D_{\!x_2}+x_2(2x_2^2z^2-z^2-1)\>+\>\D{x_1}(x_1^2-1)(x_2^2-1)z.
\end{eqnarray*}
They certify that the integral $\int_{-1}^1 f(x_1,x_2,z)\rmd x_1$ is annihilated
by $(x_2^2z^2-1)D_{\!z}+x_2^2z$ and $(x_2^2-1)(x_2^2z^2-1)D_{\!x_2}+x_2(2x_2^2z^2-z^2-1)$.
Next the operator
\[
  z(z^2-1)D_{\!z}^2+(3z^2-1)D_{\!z}+z\,-\,\D{x_2}\frac{x_2z(x_2^2-1)}{(x_2^2z^2-1)}
\]
which is a linear combination of the previous ones, 
again reveals the same ODE for the double integral.
\exampleEnd
\end{example}

\subsection{Results}\label{sec.results}
Using the above methodology and software, we have computed
differential equations for the LGFs of the fcc lattices in four, five,
and six dimensions, and rigorously proved their correctness.
Additionally, this allows the computation of the return probabilities in
the respective lattices up to very high precision.

\begin{theorem}\label{thm.fcc4}
The lattice Green's function of the four-dimensional face-centered cubic lattice
\[
  P(z) = \frac{1}{\pi^4} \int_0^\pi\int_0^\pi\int_0^\pi\int_0^\pi\frac{\rmd k_1\rmd k_2\rmd k_3\rmd k_4}
  {1-\frac{z}{6}\big(\cos k_1\cos k_2+\cos k_1\cos k_3+\dots+\cos k_3\cos k_4\big)}
\]
satisfies the differential equation
\[
  \begin{array}{l}
  (z-1) (z+2) (z+3) (z+6) (z+8) (3 z+4)^2 z^3 P^{(4)}(z) + {}\\
  2 (3 z+4) (21 z^6+356 z^5+2079 z^4+4920 z^3+3676 z^2-2304 z-3456) z^2 P^{(3)}(z) + {}\\
  6 (81z^7+1286 z^6+7432 z^5+19898 z^4+25286 z^3+11080 z^2\!-5248 z-5376) z P''(z) + {}\!\!\!\!\!\!\\
  12 (45 z^7+604 z^6+2939 z^5+6734 z^4+7633 z^3+3716 z^2+224z-384) P'(z) + {}\\
  12 (9 z^5+98 z^4+382 z^3+702 z^2+632 z+256) z P(z) = 0.
  \end{array}
\]
\end{theorem}
\begin{proof}
Here we give only an outline of the proof. The calculations in extenso 
are provided as a Mathematica notebook in the electronic supplementary 
material~\cite{Koutschan11a} (to be downloaded from
http:/$\!$/www.koutschan.de/data/fcc/).

The substitutions $\cos k_j\to x_j$ transform the integrand of the four-fold
integral to
\begin{equation}\label{eq.intg}
  f(x_1,\dots,x_4,z) = 
  \frac{1}{\big(1-\frac{z}{6}(x_1x_2+x_1x_3+\dots+x_3x_4)\big)\cdot
           \prod_{j=1}^4\sqrt{1-x_j^2}}.
\end{equation}
This expression is $\d$-finite and thus a Gr\"obner basis of the
zero-dimensional annihilating left ideal can be computed
(\texttt{ann0} in the notebook). Next, operators
$A_j(x_2,x_3,x_4,z,\D{x_2},\D{x_3},\D{x_4},\D{z})$ and
$B_j(x_1,x_2,x_3,x_4,z,\D{x_1},\D{x_2},\D{x_3},\D{x_4},\D{z})$
for $1\leq j\leq 4$ are computed, such that $A_j+\D{x_1}B_j$ is an
element in the left ideal generated by~\texttt{ann0}. This fact can be
easily tested by reducing it with the Gr\"obner basis: the remainder
being~$0$ answers the membership question in an affirmative way. In
the notebook, the $A_j$'s are collected in the variable \texttt{ann0},
and the $B_j$'s in the variable \texttt{delta1}. We conclude that
$A_1$, $A_2$, $A_3$, and $A_4$ generate an annihilating left ideal for
the integral $\int_0^\pi f(x_1,x_2,x_3,x_4,z)\rmd x_1$. In a similar
fashion, the integrations with respect to $x_2$, $x_3$, and $x_4$ are
performed, yielding a single ODE in $z$ that annihilates~$P(z)$.
\end{proof}
Note that this theorem confirms the conjectured result given
in~\cite{Guttmann09}.  Guttmann also observed that the differential equation
given in Theorem~\ref{thm.fcc4} has \emph{maximal unipotent monodromy} (MUM),
i.e., its indicial equation is of the form $\lambda^n$ and hence has only $0$
as a root, and additionally satisfies the \emph{Calabi-Yau condition}. Many
LGFs of other lattices fall into this class, too, and therefore this fact may
not seem too surprising.

In order to receive the return probability in the four-dimensional fcc
lattice, holonomic closure properties are applied to compute a differential
equation for
\[
  \frac{P(z)}{1-z}=\sum_{n=0}^\infty \bigg(\sum_{k=0}^n p_k(\0)\bigg)z^n,
\]
which in turn gives a recurrence for $f(n)=\sum_{k=0}^n p_k(\0)$:
\[
  \begin{array}{l}
  (n+2) (n+3)^2 (n+4) (35 n^2+420 n+1252) f(n) + {}\\
  (n+3) (n+4) (595 n^4+11375 n^3+79874 n^2+244384n+276024) f(n+1) + {}\\
  3 (n+4) (1015 n^5+24780 n^4+240253 n^3+1156976 n^2 + {}\\
  \qquad 2769392 n+2638272) f(n+2) + {}\\
  (3325 n^6+107100 n^5+1427695n^4+10080600 n^3+39767416 n^2 + {}\\
  \qquad 83134488 n+71984160) f(n+3) - {}\\
  4 (2065 n^6+62580 n^5+788848 n^4+5295615 n^3+19973086 n^2 + {}\\
  \qquad 40139838 n+33590844) f(n+4) - {}\\
  12 (735 n^6+25200 n^5+359282 n^4+2725632 n^3+11601091 n^2 + {}\\
  \qquad 26259960 n+24690708) f(n+5) + {}\\
  288 (35 n^2+350 n+867) (n+6)^4 f(n+6) = 0.
  \end{array}
\]
The initial values
\[
  f(0)=1,\quad f(1)=1,\quad f(2)=\frac{25}{24},\quad f(3)=\frac{19}{18},\quad f(4)=\frac{1637}{1536},\quad f(5)=\frac{549}{512}
\]
are easily (of course, not by hand!) computed by counting the number of excursions
of length up to~$5$. For the return probability 
\[
  R = 1 - \left(\lim_{n\to\infty} f(n)\right)^{-1}
\]
we need to evaluate the limit of the sequence~$f(n)$. This can be done very
accurately when knowing the asymptotics of the sequence. We apply the method
described in~\cite{WimpZeilberger85}, which has been implemented in
Mathematica~\cite{Kauers11a}, and obtain the following basis of asymptotic
solutions:
\begin{eqnarray*}
  s_1(n) & = & \frac{1}{n^2}\left(-\frac{1}{2}\right)^n 
    \left(1-\frac{5}{6n}+\frac{67}{24n^2}+\frac{1459}{144n^3}+O\Big(\frac{1}{n^4}\Big)\right),\\
  s_2(n) & = & \frac{1}{n^2}\left(-\frac{1}{3}\right)^n 
    \left(1-\frac{5}{2n}+\frac{51}{8n^2}-\frac{143}{8n^3}+O\Big(\frac{1}{n^4}\Big)\right),\\
  s_3(n) & = & \frac{1}{n^2}\left(-\frac{1}{6}\right)^n 
    \left(1-\frac{45}{14n}+\frac{4633}{392n^2}-\frac{112407}{5488n^3}+O\Big(\frac{1}{n^4}\Big)\right),\\
  s_4(n) & = & \frac{1}{n^2}\left(-\frac{1}{8}\right)^n 
    \left(1-\frac{52}{9n}+\frac{812}{27n^2}-\frac{45820}{243n^3}+O\Big(\frac{1}{n^4}\Big)\right),\\
\end{eqnarray*}
\begin{eqnarray*}
  s_5(n) & = & \frac{1}{n}\left(1-\frac{1}{n}+\frac{7}{9n^2}-\frac{7}{18n^3}+O\Big(\frac{1}{n^4}\Big)\right),
    \qquad\qquad\qquad\quad\\ 
  s_6(n) & = & 1.
\end{eqnarray*}
Obviously, the first four solutions do not significantly contribute
as they tend to~$0$ very rapidly. Thus by taking into account $s_5(n)$
and $s_6(n)$ only, and computing the asymptotic expansion to a higher order
(e.g., $30$), allows us to obtain (at least) 100 correct digits of the limit.

\begin{corollary}\label{cor.d4}
The LGF of the four-dimensional fcc lattice, evaluated at $z=1$ is
\[
  P(1)\approx 1.10584379792120476018299547088585107443954623663875285836499, 
\]
and therefore the return probability is
\[
  R\approx 0.09571315417256289673531676490121018570070881963801735768774. 
\]
\end{corollary}
\noindent
Note: Corollaries~\ref{cor.d4} and~\ref{cor.d5} have been confirmed
independently by employing the certified numerics implemented in the Maple
package \texttt{NumGfun}~\cite{Mezzarobba10}.

\begin{theorem}\label{thm.fcc5}
The lattice Green's function of the five-dimensional fcc lattice
\[
  P(z) = \frac{1}{\pi^5} \int_0^\pi\cdots\int_0^\pi\frac{\rmd k_1\rmd k_2\rmd k_3\rmd k_4\rmd k_5}
  {1-\frac{z}{10}\big(\cos k_1\cos k_2+\cos k_1\cos k_3+\dots+\cos k_4\cos k_5\big)}
\]
satisfies the differential equation
\[
  \begin{array}{l}
  16 (z-5) (z-1) (z+5)^2 (z+10) (z+15) (3 z+5) (15678 z^6+144776 z^5+{}\\
  \quad 449735 z^4+ 933650 z^3-1053375 z^2+3465000 z-675000) z^4 P^{(6)}(z)+{}\\
  8 (z+5) (3057210 z^{12}+97471734 z^{11}+1048560285 z^{10}+3939663705 z^9-{}\\
  \quad 4878146975 z^8-87265479875 z^7-304623830625 z^6-266627903125 z^5+{}\\
  \quad 254876515625 z^4-1289447109375 z^3-503550000000 z^2+{}\\
  \quad 1774828125000 z- 354375000000) z^3 P^{(5)}(z)+{}\\
  10 (27279720 z^{13}+923795772 z^{12}+11725276842 z^{11}+68439921540 z^{10}+{}\\
  \quad 148313757125 z^9-382134335775 z^8-3351125770500 z^7-7801785421250 z^6-{}\\
  \quad 3779011321875 z^5-7716298734375 z^4-39702348750000 z^3+{}\\
  \quad 3393646875000 z^2+ 23905125000000 z-5568750000000) z^2 P^{(4)}(z)+{}\\
  5 (255864960 z^{13}+7892060544 z^{12}+92744995638 z^{11}+524857986060 z^{10}+{}\\
  \quad 1350059072325 z^9\!-465440555100 z^8\!-13545524756500 z^7\!-26918293320000 z^6-{}\\
  \quad 3649915059375 z^5-77498059625000 z^4-190176960000000 z^3+{}\\
  \quad 40530375000000 z^2+45343125000000 z-13162500000000) z P^{(3)}(z)+{}\\
  5 (496679040 z^{13}+13819981248 z^{12}+149186684934 z^{11}+810956145330 z^{10}+{}\\
  \quad 2287368823475 z^9+1646226060075 z^8-8282515456375 z^7-6199228765625 z^6+{}\\
  \quad 13367806743750 z^5-110925736437500 z^4-133825053750000 z^3+{}\\
  \quad 44457862500000 z^2+ 5055750000000 z-3240000000000) P''(z)+{}\\
  10 (167064768 z^{12}+4143853440 z^{11}+40678130502 z^{10}+209673119160 z^9+{}\\
  \quad 607021304825 z^8+689643286650 z^7-135661728250 z^6+3711617481250 z^5+{}\\
  \quad 2664478321875 z^4-21210430812500 z^3-7268326875000 z^2+{}\\
  \quad 4816462500000 z- 189000000000) P'(z)+{}\\
  30 (7525440 z^{11}+163913184 z^{10}+1443544710 z^9+6925739310 z^8+{}\\
  \quad 19123388575 z^7+21336230625 z^6+36477006875 z^5+187923165625 z^4-{}\\
  \quad 55567000000 z^3-346865625000 z^2+84037500000 z+27000000000) P(z) = 0.
  \vspace{-1mm} 
  \end{array}
\]
\end{theorem}
\begin{proof}
The proof is very analogous to that of Theorem~\ref{thm.fcc4} and is given
in detail in the supplementary material~\cite{Koutschan11a}.
\end{proof}
Again, we are happy to report that our proof confirms the conjectured ODE
of~\cite{Broadhurst09}.  Remarkably enough, the indicial equation of the
differential equation presented in Theorem~\ref{thm.fcc5} is
$\lambda^5(\lambda-1)$ and hence the ODE lacks MUM. For the same reason it is
not a Calabi-Yau differential equation.

\begin{corollary}\label{cor.d5}
The LGF of the five-dimensional fcc lattice, evaluated at $z=1$ is
\[
  P(1)\approx 1.04885235135491485162956376369999275945402550465206640313845, 
\]
and therefore the return probability is
\[
  R\approx 0.04657695746384802419337442059480329107640239774632112930532. 
\]
\end{corollary}

\begin{theorem}\label{thm.fcc6}
The lattice Green's function of the six-dimensional face-centered cubic lattice
\[
  P(z) = \frac{1}{\pi^6} \int_0^\pi\cdots\int_0^\pi
  \frac{\rmd k_1\rmd k_2\rmd k_3\rmd k_4\rmd k_5\rmd k_6}
  {1-\frac{z}{15}\big(\cos k_1\cos k_2+\cos k_1\cos k_3+\dots+\cos k_5\cos k_6\big)}
\]
satisfies a differential equation of order~$8$ and with polynomials
coefficients of degree~$43$. Its leading coefficient is
\[
  \begin{array}{l}
  z^6 (z-3) (z-1) (z+4) (z+5) (z+9) (z+15)^2 (z+24) (2 z+3) (2 z+15)\\
  \times (4 z+15) (7 z+60) q(z)
  \end{array}
\]
where $q(z)$ stands for a certain irreducible
polynomial of degree~$25$, and its indicial equation is
$\lambda^6(\lambda-1)^2$.  The full equation is too long to be printed
here, but can be found in~\cite{Koutschan11a}.
\end{theorem}
\begin{proof}
The proof is very analogous to that of Theorem~\ref{thm.fcc4} and is given
in detail in the supplementary material~\cite{Koutschan11a}.
\end{proof}
As in the five-dimensional fcc lattice, the differential equation
of Theorem~\ref{thm.fcc6} lacks MUM and therefore is not Calabi-Yau.

\begin{corollary}
The LGF of the six-dimensional fcc lattice, evaluated at $z=1$ is
\[
  P(1)\approx 1.02774910062749883985936367927396850209243990900114872425172, 
\]
and therefore the return probability is
\[
  R\approx 0.02699987828795612426936417542619638021612262676239501413843. 
\]
\end{corollary}

We want to conclude this section with an overview of our results
concerning the return probabilities, which reveals an interesting
dependence on the dimension of the lattice:
\begin{center}
\begin{tabular}{|c|l|}\hline
Dimension & Return Probability \\ \hline
2 & 1 \\
3 & 0.256318236504649 \\
4 & 0.095713154172563 \\
5 & 0.046576957463848 \\
6 & 0.026999878287956 \\ \hline
\end{tabular}
\end{center}

\section{Outlook}
While the calculations for Theorem~\ref{thm.fcc4} and
Theorem~\ref{thm.fcc5} are performed in a few minutes respectively hours,
it was a major effort of several days to compute the certificates that
prove Theorem~\ref{thm.fcc6}; they are several hundred MegaBytes in
size. With the methods described in this paper and with the current
hardware, it is completely out of the question to attack the fcc
lattice in seven dimensions.  An interesting question is whether the
pattern that showed up in dimensions four to six continues. This would
suggest a differential equation of order~$10$ with indicial equation
$\lambda^7(\lambda-1)^3$. But who knows?

For the three corollaries we computed the approximations for the return
probabilities with more than one hundred valid digits.  But we have no clue
what their exact values are. Banderier evaluated these numbers up to
several thousand digits~\cite{Banderier11}, but also he was unable to identify
the closed forms. So we leave these questions open, as a challenge for future
research.

\subsection*{Acknowledgments}
Anthony Guttmann aroused my interest in computing LGFs during his talk
at the 65th S\'{e}minaire Lotharingien de Combinatoire in Strobl,
Austria.  I would like to thank him for his encouragement and advice.
I am grateful to my colleague Manuel Kauers for interesting
discussions and for pointing me to the trick of multi-step guessing.
This research was carried out while the author was employed at the
Research Institute for Symbolic Computation (RISC), Linz, Austria.
The author was supported by the Austrian Science Fund (FWF): P20162-N18.

\bibliographystyle{unsrt}
\bibliography{literatur}

\end{document}